 \documentclass[11pt,a4paper]{amsart}
\usepackage{amsmath,amsthm,amsfonts,amssymb,amsbsy,latexsym,graphicx}

\DeclareMathOperator{\Tr}{Tr}
\DeclareMathOperator{\End}{End}

\DeclareMathOperator{\F}{\mathcal{F}}
\DeclareMathOperator{\C}{\mathcal{C}}
\DeclareMathOperator{\rank}{\text{Rank}}

\DeclareMathOperator{\I}{\mathcal{I}}

\DeclareMathOperator{\supp}{supp}

\usepackage{multirow,array}
\usepackage{arydshln}
\usepackage{acronym}
\usepackage{hyperref}

\newtheorem{theorem}{Theorem}[section]
\newtheorem{corollary}[theorem]{Corollary}
\newtheorem{lemma}[theorem]{Lemma}
\newtheorem{proposition}[theorem]{Proposition}

\allowdisplaybreaks

\theoremstyle{definition}

\newtheorem*{theorem*}{Theorem}

\usepackage{mathrsfs}

\numberwithin{equation}{section}

\usepackage{xcolor}
\usepackage{dsfont}
 
 \usepackage[margin=1.2in]{geometry}       
 \usepackage{microtype}                    
 \usepackage[utf8]{inputenc}               
 \usepackage{xcolor}                       
 
 \definecolor{primary}{RGB}{31, 78, 121}    
 \definecolor{secondary}{RGB}{110, 110, 110}
 
 \usepackage{booktabs}                     
 \usepackage[most]{tcolorbox}              
 
 \newtcolorbox{infobox}{
 	colback=primary!5, 
 	colframe=primary, 
 	boxrule=0.5pt, 
 	arc=4pt,
 	title=\textbf{Note:}
 }

\begin{document}
\title[Smoothness of the Radon-Nikodym derivative]{Convolution of $\chi$-orbital Measures on Complex Grassmannians}
\author{Mahmoud Al-Hashami and Boudjem\^aa Anchouche}
\address{University of Technology and Applied Sciences, Nizwa, \ Oman.\newline
Department of Mathematics, Kuwait University}
\email{mahmoud.alhashami@utas.edu.om\\
boudjemaa.anchouche@ku.edu.kw}
\date{
}

\begin{abstract}
Let  {\scriptsize $SU(p+q)/S\left(U(p)\times U(q) \right)$} be the Grassmannian of complex $p$-dimensional subspaces of $\mathbb{C}^{p+q}$, where $p$ and $q$  are integers such that $p\geq q \geq 2$. 
The aim of this paper is to extend the main result in \cite{anchouche2}, \cite{Alhashami} to the case of convolution of $\chi$-orbital measures  where $\chi$ is a character of $S\left(U(p)\times U(q) \right)$. More precisely, we give sufficient conditions for the $\C^{\nu}$-smoothness of the Radon Nikodym derivative $$f_{ a_{1},...,a_{r}, \chi}=d\left(\mu_{a_1, \chi}\ast...\ast\mu_{a_r, \chi}\right)  /d\mu_{{SU(p+q)}}$$ of the convolution
$\mu_{a_1, \chi}\ast...\ast\mu_{a_r, \chi}$  with respect to the Haar measure $\mu_{\mathsf{SU(p+q)}}$ of $SU(p+q)$, where $\mu_{a_j, \chi}$ are some orbital measures defined below.
\end{abstract}
\maketitle
\tableofcontents

\section{Introduction} \label{intro}

Let $M=U/K$ be a compact symmetric space and let 
$a_1, a_2, \cdots, \, a_r $ be elements in $U\backslash N_U(K)$, where $N_U(K)$ is the normalizer of $K$ in $U$, and consider the following linear functionals
defined by
\[ \I_{a_j}\left(f \right) =\iint_{K\times K} f\left(k_1a_jk_2 \right) d\mu_K \left(k_1\right)d\mu_K\left( k_2\right),  \;\;\;\;\;\;\; f\in \C(U)    , \,j=1,...,r,\]
where $\C\left(U\right)$ is the  space of continuous functions on $U$, and $\mu_K$ is the Haar measure of the Lie group $K$ normalized by ${ \mu_K\left(K \right)} =1$.  By Riesz representation theorem, to each of the functionals $\I_{a_j}$, $j=1,...,r$, corresponds a singular measure $\mu_{a_j}$ in $U$  supported on $Ka_jK$. Several authors considered the problem of the regularity of the convolution of the orbital measures $\mu_{a_j}$. In \cite{RaZonal}, Ragozin proved that the convolution of $r$ orbital measures is absolutely continuous with respect to the normalized  Haar measure of $U$ if $r \geq \dim M$.
 The absolute continuity of the convolutions of the  measures \(\mu_{a_j}\)  with respect to the Haar measure \(\mu_U\) of \(U\),  has also been discussed in \cite{AG}.
  The \(L^2\) regularity of the Radon-Nikodym derivative of the convolution of the measure  \(\mu_{a_j}\), where $M=SU(2)/SO(2)$,  with respect to the Haar measure of $SU(2)$, has been studied in \cite{AGP}.  
The smoothness of the Radon-Nikodym derivative of the convolution of the measures \(\mu_{a_j}\) on some symmetric compact spaces   $U/K$  with respect to the normalized Haar measure of \(U\),   was considered in \cite{Alhashami} and \cite{anchouche2}.  
  The convolution of orbital measures on arbitrary non-compact symmetric spaces $G/K$ was studied in \cite{anchouche-noncompact case}. Specifically, the paper examines the  $L^2$-regularity (resp. $C^\nu$-smoothness) of the Radon-Nikodym derivative of the convolution $\mu_{a_1}\ast\cdots \ast\mu_{a_r}$ with respect to a fixed left Haar measure $\mu_G$ on $G$.

It is known that if a result is true for Grassmannians, it is likely to be true for all compact symmetric spaces. Following this guideline, we considered in \cite{Alhashami} the question of smoothness of the Radon-Nikodym derivative of the convolution of the orbital measures $\mu_{a_j}$ on complex Grassmannians.  In this paper, we consider the case of the complex Grassmannian $M=U/K$, where $U=SU\left(p+q \right) $ and $K=S\left(U(p)\times U(q) \right) $, with $p$ and 
$q$ being integers such that $p\geq q\geq 2$, and we introduce the following twisted complex valued functionals: 
$$
\I_{a_j,\chi}\left(f \right)=\int_K\int_K f\left(k_1a_jk_2 \right)\chi(k_1k_2)^{-1}  d\mu_K(k_1) d\mu_K(k_2), \, \, \,  f \in \C \left(U \right) , \,j=1,...,r,
$$
where $\C\left(U\right)$ is the  space of continuous functions on $U$,  
and $$
\chi: K \longrightarrow \mathbb{C}
$$
is a character. Again, by Riesz representation theorem, to the functional $\I_{a_j,\chi}$ corresponds a $\mathbb{C}$-valued regular Borel measure $\mu_{a_j, \chi}$, $j=1,2,...,r$. 
Since \\$\supp(\mu_{a_j,\chi}) = Ka_jK$ for $j=1,\dots,r$ 
and $Ka_jK$ has empty interior in $U$, 
the measures $\mu_{a_j,\chi}$ are singular. 

In the first part of this paper, we give a sufficient condition for $\mu_{a_1, \chi}\ast \cdots \ast \mu_{a_r, \chi}$ to be absolutely continuous with respect to the normalized Haar measure $\mu_U$ of $U$, obtained in section \ref{Absolute Continuity}. 
In the second part, for each given  positive integer $\nu$, a lower bound 
$C\left(p,q,\nu \right)$,  as defined below, is provided such that for every 
$r\geq C(p,q,\nu)$, the Radon-Nikodym derivative   
\[
\frac{d\big(\mu_{a_1, \chi} \ast \cdots \ast \mu_{a_r, \chi}\big)}{d\mu_U}
\]
is in $\mathcal{C}^{\nu}(U)$, the space of $\nu$-times continuously differentiable functions on $U$, where $\mu_U$ is the Haar measure on $U$, normalized by $\mu_U(U)=1$. More precisely, let $\nu$ be a positive integer and \begin{equation*}  
C\left(p,q,\nu \right):=\max\left(\left\lfloor\frac{(p+q)^2+2\nu+q(2p-1) }{2p-q}\right\rfloor+1, 2pq\right). 
\end{equation*}
The aim of this paper is to prove the following

\begin{theorem*}[Main Theorem]\label{main} 
	Let $\nu$ be a positive integer. If 
	\[
	r\geq C\left(p,q,\nu \right),
	\]
	then 
	$$
	\frac{d\big(\mu_{a_1, \chi}\ast \cdots \ast \mu_{a_r, \chi}\big)}{d\mu_U} \in \C^{\nu}\left(SU(p+q)\right).
	$$
\end{theorem*} 
The paper is organized as follows. 
In section \ref{Spherical Representations}, we collect some results about $\chi$-spherical representations and $\chi$-spherical functions which will be needed in the paper. In section \ref{Absolute Continuity}, we give a sufficient condition for the absolute continuity of the convolution of the  $\chi$-orbital measures with respect to the normalized Haar measure on $U$.
In section \ref{Fourier Transform},	we study the Fourier transform of the Radon-Nikodym derivative of the convolution of $\chi$-orbital measures with respect to the Haar measure on $U$.
In section \ref{Complex Grassmannians}, we review some of the main properties of restricted roots on complex Grassmannians.
In section \ref{Proof of the Main Theorem}, we give a proof of the main result of the paper.

\section{The $\chi_l$-Spherical Representations}\label{Spherical Representations}
Let $G$ be a  non-compact, connected, simple  Lie group, and  $\mathfrak{g}$ its Lie algebra. Let $$ \mathfrak{g}= \mathfrak{k}\oplus \mathfrak{p}, $$  be  a Cartan decomposition of the Lie algebra $\mathfrak{g}$.   It is known that  $$ \mathfrak{u}= \mathfrak{k}\oplus \sqrt{-1}\mathfrak{p} $$ is a compact real form of $\mathfrak{g}_\mathbb{C}$, the complexification of $\mathfrak{g}$. Let $ \mathfrak{ a}$ be a maximal abelian  subspace of $\mathfrak{p}$ and $\mathfrak{h}$   a Cartan subalgebra of $\mathfrak{g}$ containing $\mathfrak{a}$. Then 
$ \mathfrak{h}= \mathfrak{m} \oplus \mathfrak{a}$, where $\mathfrak{m}= \mathfrak{ h} \cap \mathfrak{k}$. 
{    Let $U$ be the simply connected, simple  Lie group with Lie algebra $\mathfrak{u}$ and let $K$ be a Lie group with Lie algebra $\mathfrak{k}$.}

Since every character $\chi$ of $K$ is trivial on the commutator subgroup $\left[K,K \right]$, and since $K=Z\left(K \right)\left[K,K \right] $, $\mathfrak{k}= [\mathfrak{k},\mathfrak{k}]\oplus\mathfrak{z}(\mathfrak{k})$, where $\mathfrak{z}(\mathfrak{k})$ is   the center of $\mathfrak{k}$, 
$Z(K)$  is the center of $K$, we see that if $\mathfrak{z}(\mathfrak{k}) $ is trivial, then every character is trivial. Therefore, we will assume in what follows that $\mathfrak{z}(\mathfrak{k}) $ is not trivial, hence $K$ is not semisimple and therefore $U/K$ is Hermitian.

Let $Z$ be a non-zero  element of $\mathfrak{z}(\mathfrak{k})$  such that $\exp\left( tZ\right) \in Z\left( K\right) $ and \\ $\exp(tZ) \in \left[K,K \right]$ if and only if $t \in 2 \pi \mathbb{Z}$.  For each integer $l$, put
\[
\chi_l\left( x \right) =
\begin{cases}
1 &\text{ if \( x \in \left[K, \, K \right] \);}  \\ \\
\exp\left(itl \right)  &\text{ if \( x=\exp\left(tZ \right)  \).}
\end{cases}
\] 
It can be seen that for each integer $l$, $\chi_l$ is a character of $K$, moreover every character of $K$ is of that form, i.e., the family $\left( \chi_l\right)_{l \in \mathbb{Z}}$ exhausts all the characters of $K$, more precisely, we have the following
\begin{proposition}[\cite{Schlichtkrull}] 
	For every integer $l$, $\chi_l$ is a character of $K$. Moreover, every character of $K$ is of the form $\chi_l$ for some integer $l$.	
\end{proposition}
Let $\pi_{\lambda}:U \longrightarrow GL(E_{\lambda})$ be an irreducible unitary representation of $U$ with highest weight $\lambda$. For a non-trivial character $\chi_l$ of $K$, let
\begin{equation}\label{fixed vectors by K}
E_{\lambda}^{l}=\left\lbrace X \in E_{\lambda} \, \arrowvert \,  \pi_\lambda\left(k \right)\left(X \right) =\chi_l\left( k\right)X    \text{ for all } k \in K\right\rbrace. 
\end{equation}
The representation  $\left(  \pi_\lambda, E_{\lambda}\right)$ 
of $U$ is said to be $\chi_l$-spherical if $E_{\lambda}^{l} \neq 0$. The spherical representations correspond to the case  $l=0$. \\

Fix an integer $l$ in $\mathbb{Z}$, from now on we will denote by $\widehat{U}_l$ the set of $\chi_l$-spherical irreducible unitary representations of $U$ and $\Lambda^+(U)$   the set of highest weights of irreducible representation of $U$. 
It is  well known that if $U/K$ is a symmetric space and $\left(\pi_\lambda, E_\lambda\right)$ is a $\chi_l$-spherical representation, then $\dim_{\mathbb{C}} E_{\lambda}^l =1$. The set of highest weights of $\chi_l$-spherical representations of $U$ will be denoted by $\Lambda_l^+\left( U\right) $.\\ 

Recall that a function $\varphi$ on $U$ is called (an elementary) $\chi_l$-spherical  function if  it is not identically zero, and if it satisfies the following conditions:
 \begin{align}\label{spherical-def1}
\varphi(k_1uk_2)= \chi_l(k_1k_2)^{-1} \varphi(u)  , \text{ for all } u\in U ,\; k_1,k_2 \in K,
\end{align}
and 
\begin{align}\label{spherical-def2}
\int_K \varphi( u_1ku_2 ) \chi_l(k) \;d\mu_K(k)= \varphi(u_1)\varphi(u_2), \text{ for all } u_1,u_2\in U.
\end{align}
Actually it can be proved that $\left(\ref{spherical-def1} \right) $ is a consequence of $\left(\ref{spherical-def2} \right) $, see \cite{{Ho-Olafsson}}.\\

For  $\lambda \in \Lambda_l^+\left( U\right)$, let
 $\left(.,.\right)_{\lambda}$ be a $U$-invariant inner product on $E_\lambda$, i.e., 
 \[
  \left(\pi_\lambda\left(u\right)X, \pi_\lambda\left(u\right)Y \right)_{\lambda}=\left(X,Y \right)_\lambda. 
 \]
 for all $u \in U$ and $X, Y\in E_\lambda$,
  and let $\| \cdot    \|_\lambda$ the corresponding norm. Choose an orthonormal basis $\lbrace e_1^\lambda , \cdots , e^\lambda_{d_\lambda}\rbrace$   of $ E_\lambda$, such that $ e_1^\lambda \in E_\lambda^l$, where $ d_{\lambda}=\dim E_{\lambda}$. \\

It is well known that $\psi_{\lambda,l}(u) = \left( e_1^\lambda,\pi_\lambda(u)  e_1^\lambda  \right)_{\lambda}$ is an elementary $\chi_l$- spherical function and conversely  any elementary $\chi_l$- spherical function on $U$ is of this form.

\section{Absolute Continuity of Convolutions of $\chi_l$-Orbital Measures}\label{Absolute Continuity}

Let $a \in U\backslash N_U(K)$, and let $\I_{a,\chi_l}$  be as in section $\ref{intro}$. By the  Riesz representation theorem, to the functional $\I_{a,\chi_l}$ corresponds a unique $\mathbb{C}$-valued regular Borel measure $\mu_{a, \chi_l}$ on $U$, such that 
$$
\I_{a,\chi_l}\left(f \right) =\int_{U}f(g) d\,\mu_{a, \chi_l}(g),\, \, \, f\in \C\left(U\right).
$$ 
Fix a set of points $a_1, \cdots, a_r$ in $U\backslash N_U(K)$ and consider the linear functional
\begin{small}
$$
\I_{a_1,a_2,\cdots, a_r, \chi_l}\left( f\right) =\int_K \cdots\int_K f\big(k_1a_1k_2\cdots k_ra_rk_{r+1}\big)\chi_l\left(k_1\cdots k_{r+1} \right)^{-1}d\mu_K(k_1) \cdots d\mu_K(k_{r+1}). 
$$
\end{small}
Again, by the Riesz representation theorem, to the functional $\I_{a_1,a_2,\cdots, a_r, \chi_l}$ corresponds a measure $\mu_{a_1,\cdots, a_r, \chi_l}$, such that
$$
\I_{a_1,a_2,\cdots, a_r, \chi_l}\left(f \right) =\int_U f(g)d\,\mu_{a_1,\cdots, a_r, \chi_l}\left( g\right)
$$
for all $f$ in $\C\left(U\right)$.  
\begin{proposition}
$$
\mu_{a_1,a_2,\cdots, a_r, \chi_l}=\mu_{a_1, \chi_l}\ast \cdots \ast \mu_{a_r, \chi_l}.
$$
\end{proposition}
\begin{proof} Let $ f \in \C\left(U\right)$ and let $r=2$. Since $\mu_K$ is normalized, we have
 
  	\begin{footnotesize}
  	\begin{align*}
	\int_U f(g) \;d\mu_{a_1, \chi_l}\ast \mu_{a_2, \chi_l} (g)
	&=\int_U\int_Uf(g_1g_2) d\mu_{a_1, \chi_l}(g_1)d\mu_{a_2, \chi_l}(g_2)\\&= \int_K \int_K \int_K\int_K  f(k_1a_1k_2k_3a_2k_4) \chi_l^{-1}(k_1k_2k_3k_4) \;d\mu_K(k_1)d\mu_K(k_2)d\mu_K(k_3)d\mu_K(k_4)
	\\&= \int_K \int_K \int_K \left(\int_K  f(k_1a_1k_2k_3a_2k_4) \chi_l^{-1}(k_1k_2k_3k_4) \;d\mu_K(k_2)\right)d\mu_K(k_1)d\mu_K(k_3)d\mu_K(k_4) 
	\\&= \int_K\int_K \int_K \left(\int_K  f(k_1a_1k_5a_2k_4) \chi_l^{-1}(k_1k_5k_4) \;d\mu_K(k_5)\right)d\mu_K(k_1)d\mu_K(k_3)d\mu_K(k_4) \\
	&= \int_K \int_K \int_K f(k_1a_1k_2a_2k_3) \chi_l^{-1}(k_1k_2k_3) \;d\mu_K(k_1)d\mu_K(k_2)d\mu_K(k_3)\\
	&= \int_U f(g) \;d\mu_{a_1,a_2,\chi_l}(g).
	\end{align*}
  	\end{footnotesize}
   	
	By induction, we get
	$$
	\int_U f(g) d \big(\mu_{a_1, \chi_l}\ast \cdots \ast \mu_{a_{r}, \chi_l}\big)(g)
	=   \int_U f(g) \;d \mu_{a_1,a_2,\cdots, a_{r}, \chi_l}, \label{Convolution equality}
	$$
	for all $f \in \C\left(U\right)$. Hence the proposition.
\end{proof}
\begin{lemma}\label{Characteristic lemma}
For every Borel subset $\mathscr{B}$ of $U$, we have 

	\begin{itemize}  
	\item[(i)] 
	$
	\mu_{a, \chi_l}(\mathscr{B})=\int _{U}{\mathds{1}}_\mathscr{B}(x)\,d \mu_{a, \chi_l}(x) =\int\limits_{K}\int\limits_{K}{\mathds{1}}_\mathscr{B}\left(  k_{1}%
	ak_{2}\right) \chi_l\left(k_1k_2 \right)^{-1}   d\mu_K\left(  k_{1}\right)  d\mu_K\left(  k_{2}\right),
	$\\
	and \\
	
	\item[(ii)] 
	\begin{small}
	$
	\int_U {\mathds{1}}_\mathscr{B}(x) \;d\mu_{a_1, \dots ,  a_r, \chi_l}=\int_K  \dots\int_K {\mathds{1}}_\mathscr{B}\left(k_1a_1k_2\cdots k_ra_rk_{r+1} \right)\chi_l(k_1 \cdots k_{r+1})^{-1} d\mu_{K}\left(k_1 \right)  \cdots d\mu_{K}\left( k_{r+1}\right) ,
	$
	\end{small}
	\end{itemize}

where  ${\mathds{1}}_\mathscr{B}$ is the characteristic function of $\mathscr{B}$.

\end{lemma}

\begin{proof}
	
	Consider the continuous mapping $ \tau_a:K\times K \longrightarrow U$ such that $$ \tau_a(k_1,k_2)=k_1ak_2.$$ Let $\lambda_{a}$ be the push-forward of the product measure $\mu_K \times \mu_K$ on $K \times K$ under the map $\tau_a$, i.e., for a Borel subset $\mathscr{B}$ of $U$,
	$$
	\lambda_{a}\left( \mathscr{B}\right) =\left( \tau_a\right) _*\left( \mu_K \times \mu_K\right)\left(\mathscr{B} \right)= \big(\mu_K \times \mu_K\big)\left(\tau_a^{-1}\left( \mathscr{B}\right)  \right).   
	$$
	Hence
	$$
	\lambda_{a}\left( \mathscr{B}\right)=\int\limits_{K}\int\limits_{K}{\mathds{1}}_\mathscr{B}\left(  k_{1}%
	ak_{2}\right)   d\mu_K\left(  k_{1}\right)  d\mu_K\left(  k_{2}\right) .
	$$
	By   (Corollary 1 \cite{Stromberg}), $ \lambda_{a}$ is regular.
	Define the complex measure $ \lambda_{a, \chi_l}$ on $U$ by  \begin{align*}
	\lambda_{a, \chi_l}(\mathscr{B})= \int\limits_{K}\int\limits_{K}{\mathds{1}}_\mathscr{B}\left(  k_{1}%
	ak_{2}\right) \chi_l\left(k_1k_2 \right)^{-1}   d\mu_K\left(  k_{1}\right)  d\mu_K\left(  k_{2}\right).
	\end{align*}
	Fix a Borel set $\mathscr{B}$, and let $\epsilon>0$. Since $\lambda_{a}$ is regular, there exist an open set $O$ and a compact set $C$ of $U$ such that $C\subseteq \mathscr{B} \subseteq O$ and 
	 \begin{align*}
	\lambda_{a}(\mathscr{B} \backslash C) < \epsilon \;\;\;\; \text{ and } \;\;\;\;  \lambda_{a}(O \backslash \mathscr{B} )< \epsilon. 
	\end{align*}   
    Therefore
    \begin{align*}
	|\lambda_{a, \chi_l}|(\mathscr{B} \backslash C)&= \sup   \sum _{i=1}^n\left|\lambda_{a, \chi_l} ( \mathscr{B}_j)\right| \\&\leq \sup   \sum _{i=1}^n\lambda_{a} ( \mathscr{B}_j) \\ &=\lambda_{a}(\mathscr{B} \backslash C) \\&<\epsilon,
	\end{align*} where the supremum is taken over all partitions of $\mathscr{B}\backslash C$  into a finite number of Borel sets $\mathscr{B}_1,\dots \mathscr{B}_n$. Similarly $ |\lambda_{a, \chi_l}|(O \backslash \mathscr{B})< \epsilon$. Hence $\lambda_{a, \chi_l}$ is regular. A similar  argument as in (Theorem 2, \cite{Stromberg}), shows that for any Borel set $\mathscr{B}$, \\ $\lambda_{a, \chi_l}(\mathscr{B}) = \mu_{a, \chi_l}(\mathscr{B})
	$, which proves part $(i)$ of the Lemma. Part $(ii)$ of the Lemma  follows from  ( \cite{Stromberg}, Theorem 2).
\end{proof}

\begin{proposition}  \label{Zonel-proposition}
Let $ a_i\in U\backslash N_U(K)$, $i=1,\dots, r$ and   $\psi^r: K^{r+1}\longrightarrow U$, be such that  \begin{align*}
\psi^r (k_1,\dots k_{r+1}) = k_1a_1k_2\dots k_ra_rk_{r+1} .
\end{align*}
Then $ \rank \psi^r \geq \min\left(\dim U, r+\dim K\right)$, 
except on a proper analytic subvariety of $K^{r+1}$.
\end{proposition}
\begin{proof} By 
\cite{Wolf3}, the pair $(U,K)$ satisfies condition $(2.1)$ in \cite{RaZonal}. So the Proposition follows from (\cite{RaZonal}, Proposition 2.2).
\end{proof}

\begin{theorem} \label{Rgozin Theorem}
Let $(U,K)$ and  $\mu_{a_i, \, \chi_l}$, $i=1,\dots , r$ be as above. If $r \geq \dim U /K$ then $ \mu _{a_1, \dots , a_r, \, \chi_l}$ is absolutely continuous with respect to the Haar measure $\mu_U$ of  $U$. 
\end{theorem} 
\begin{proof}
 Let $S\subseteq U$ be a measurable set and $ {\mathds{1}}_S$ its characteristic function. Assume that $ \mu_U(S)=0$, it will be shown that $ \mu _{a_1, \dots , a_r, \, \chi_l}(S)=0$. By Proposition \ref{Characteristic lemma} we have \begin{small}
 \begin{align*}
\left\vert \mu _{a_1, \dots , a_r, \chi_l}(S) \right\vert & =\left\vert\int _K\dots\int_K  {\mathds{1}}_S( k_1a_1k_2\dots a_rk_{r+1}) \chi_l (k_1\dots k_{r+1})\;d\mu_K(k_1)\dots d\mu_K(k_{r+1})\right\vert \\
& \leq \int _K\dots\int_K  {\mathds{1}}_S( k_1a_1k_2\dots a_rk_{r+1}) \;d\mu_K(k_1)\dots d\mu_K(k_{r+1})\\
&\leq \mu_{K^{r+1}}({(\psi^r)}^ {-1}(S)),
\end{align*}
 \end{small} where $\mu_{K^{r+1}}$ is the Haar measure of $K^{r+1}$.
To finish the proof we need to show that  
$$
\psi^r_{*}\left(\mu_{K_{r+1}} \right) \left( S\right)= \mu_{K^{r+1}}({(\psi^r)}^ {-1}(S))=0.
$$
By  Proposition \ref{Zonel-proposition},  $ \psi^r$ has full rank, i.e., $\rank \psi^r = \dim U = \dim K + \dim U/K$,  except on a proper analytic subvariety of $K^{r+1}$. This subvariety has measure zero, since it has smaller dimension than $K^{r+1}$. \\
Assume that $K^{r+1} $ (resp. $ U$) has dimension $s$ (resp. $ m$).  Note that $s\geq m$ and $ (\psi^r)^{-1} (S)$ is closed.

Consider $p'\in  (\psi^r)^{-1}(S)$  and let $p=\psi^r(p')$. {   By the submersion theorem }
there are charts  $ (\phi , H)$  centered at $p'$ and a chart  $(\varphi
, H')$)  centered at $p$ such that \\ $\psi^r(H)\subseteq H'$,  $$\varphi
 \circ \psi^r\circ \phi^{-1}(x,y)=x, \;\;\;\;\; \text { for all } x\in \mathbb{R}^m, \;y\in \mathbb{R}^{s-m},\; (x,y)\in \phi(H).$$
Let $\lambda_m$ be the Lebesgue measure of $\mathbb{R}^m$.     Since  the  Haar   measure on any Lie group is mutually absolutely continuous with respect to  Lebesgue  measure in any coordinate patch we have  $$\lambda_m(\varphi
 (H'\cap S)=0.$$ 
   Let $A\subseteq \mathbb{R}^s  $ be such that $$ A= \phi(H\cap (\psi^r)^{-1}(S)), $$ It follows that
   \begin{align*}
   A\subseteq \lbrace (x_1, \dots , x_m, x_{m+1}, \dots , x_s) \in \mathbb{R}^s \mid\; (x_1, \dots , x_m)\in \varphi
 (H'\cap S)\rbrace.
      \end{align*} By Fubini's Theorem, we have 
      \begin{align*}
      \lambda_{s} (A) &= \int_{\mathbb{R}^s} {\mathds{1}}_A(x) \;d\lambda_s(x)\\&
      = \int_{\mathbb{R}^l}\int_{\mathbb{R}^m} {\mathds{1}}_A(x,y) \;d\lambda_m(x)d\lambda_l(x) \\&\leq  \int_{\mathbb{R}^l}\int_{\mathbb{R}^m} {\mathds{1}}_{\varphi
 (H'\cap S)}(x) \;d\lambda_m(x)d\lambda_l(x)\\&= \int_{\mathbb{R}^l} \lambda_m( \varphi (H'\cap S)) \;d\lambda_l(x)\\&= 0.
      \end{align*}
Thus  for any  $x\in K^{r+1}$, there exists a  coordinate neighborhood $H \subseteq K^{r+1}$ of $x$ such that  $$\mu_{K^{r+1}}( H\cap (\psi^r)^{-1}(S))= 0.$$   
 By compactness of $K^{r+1}$, we can find a finite coordinate neighborhoods $\lbrace H_i \rbrace_{i=1}^d$ such that $$ (\psi^r)^{-1}(S) \subseteq \bigcup _{i=1}^d H_i,$$  so we have $$ (\psi^r)^{-1}(S) = \bigcup _{i=1}^d (H_i \cap (\psi^r)^{-1}(S)).$$ Therefore  $$\mu_{K^{r+1}} (\psi^r)^{-1}(S)) \leq \sum_{i=1}^d \mu_{K^{r+1}}( H_i\cap (\psi^r)^{-1}(S))=0.$$     
      Hence the Theorem.    
\end{proof}

\section{Fourier Transform of the Radon-Nikodym Derivative $f_{a_1,\dots ,a_r,\chi_l}$}	\label{Fourier Transform}
Let $\chi_l:K \longrightarrow GL\left(E_{\chi_l} \right) $ be a representation of $K$. Define an action of $K$ on $U \times E_{\chi_l}$ from the right by
\[
\left(u,X \right)k=\left(uk, \chi_l\left(k^{-1}\right)  X \right), \, \, \, \,\, \,\,  u \in U, \, \, X \in E_{\chi_l}. 
\]
Define an equivalence relation $\mathcal{R}$ on $U \times E_{\chi_l}$ by
\[
\left(u_1,X_1 \right) \mathcal{R} \left( u_2,X_2\right) \Leftrightarrow  \exists\, k \in K \text{ such that } u_2=u_1k \text{ and } X_2=\chi_l\left( k^{-1}\right)X_1. 
\]
The equivalence class of an element $\left(u,X \right) $ of $U \times E_{\chi_l}$ will be denoted by $\left[ u,X\right]$ and the set of equivalence classes by ${\mathbb{E}_{\chi_l}}=U\times_K E_{\chi_l}$, i.e.,
\[
{\mathbb{E}_{\chi_l}}=U\times_K E_{\chi_l}=\left\lbrace  \left[ u,X\right]\,\, \arrowvert \, \, u\in U \text{ and } X \in E_{\chi_l} \right\rbrace. 
\]
Then $\mathbb{E}_{\chi_l}$ is a complex vector bundle over $U/K$ of rank equal to  $\dim_{\mathbb{C}}E_{\chi_l} $.
In case $\chi_l$ is a character, i.e., $\dim_{\mathbb{C}} E_{\chi_l} =1$, or equivalently $E_{\chi_l} \cong \mathbb{C}$, then $\mathbb{E}_{\chi_l}$ is a line bundle, which we will denote, following tradition, by $\mathbb{L}_{\chi_l}$. It is easy to see that the smooth sections $\Gamma\left(U/K, \, \mathbb{L}_{\chi_l}\right) $ of $\mathbb{L}_{\chi_l}$ can be identified with smooth functions $f$ on $U$ satisfying 
$$
f\left( uk\right)=\chi_l(k)^{-1}f\left(u \right), \mbox{ for all } k\in K \mbox{ and } u \in U.
$$
Let \begin{align*}
L^2 \left(U/K, \, \mathbb{L}_{\chi_l}\right)= \left\{  f\in L^2(U) \mid f( uk)=\chi_l(k^{-1})  f(u), \text{ for all } k\in K  \right\},
\end{align*}
and let $L^2 \left(U//K, \, \mathbb{L}_{\chi_l}\right)$ be the subspace in  $L^2 \left(U/K, \, \mathbb{L}_{\chi_l}\right)$ of functions satisfying  $$f(k_1uk_2) =\chi_l( k_1k_2 )^{-1} f(u), \,\,\,\,\,\,\,\,\, \text{ for all }  k_1, k_2\in K.$$
 
Let $\lambda \in \Lambda^+_l\left(U\right)$ and, as above, let $\pi_{\lambda}:U \longrightarrow GL(E_{\lambda})$ be the irreducible unitary representation of $U$ with highest weight $\lambda$. 
For $f \in L^1(U)$, consider the bilinear form
\begin{align*}
 \Psi_\lambda :E_{\lambda} \times E_{\lambda}&\longrightarrow  \mathbb{C}  \\
(X,Y)& \longmapsto \Psi_\lambda (X,Y) =\int_U f(u)\,\bigl( \pi_{\lambda}(u)\,X, Y\bigr)_\lambda \, d\mu_U(u),\;\;\;\;\;X, Y \in E_{\lambda}.
\end{align*}
 
 Since
\[
|\Psi_\lambda(X,Y)| \;=\; \left|\int_U f(u)\,\bigl( \pi_{\lambda}(u)\,X, Y \bigr)_\lambda \, d\mu_U(u)\right|
\;\le\; \|X\|_{\lambda} \,\|Y\|_{\lambda} \,\|f\|_{L^1(U)},
\]
it follows that $\Psi_\lambda$ is continuous. By Riesz representation theorem, we deduce that there exists a continuous linear operator 
\[
\mathcal{F}\bigl(f\bigr)(\lambda)\colon E_{\lambda} \,\longrightarrow\, E_{\lambda}
\]
such that
\[
\bigl( \mathcal{F}\bigl(f\bigr)(\lambda)(X), Y\bigr)_\lambda \;=\; 
\int_U f(u)\,\bigl( \pi_{\lambda}(u)\,X, Y\bigr)_\lambda \,d\mu_U(u).
\]
The operator $\mathcal{F}\bigl(f\bigr)(\lambda)$ is called the \emph{Fourier transform} of $f$ and for simplicity we write
\[
 \mathcal{F}\bigl(f\bigr)(\lambda)\left( X\right) =\int_U f(u)\, \pi_{\lambda}(u)\left( X\right) \,d\mu_U(u).
\]

Recall that for a nontrivial character $\chi_l$ of $K$, we denoted by $E_\lambda^l$ the subspace of $E_\lambda$ defined in \eqref{fixed vectors by K}. 
Similarly, as above, there exists a continuous linear operator 
\[
P_{l,\lambda}\colon E_{\lambda}\,\longrightarrow\,E_{\lambda}
\]
such that
\[
 \bigl( P_{l,\lambda}(X),\,Y\bigr)_\lambda 
\;=\;
\int_K \chi_l\bigl(k^{-1}\bigr)\,\bigl(\pi_\lambda(k)\,X,\; Y\bigr)_\lambda \, d\mu_K(k) ,\,\,\,\,\, X,Y\in E_\lambda.
\]
For simplicity, we will write
\begin{align*}
	P_{l, \lambda} (X) =\int_K \chi_l(k^{-1}) \pi_\lambda(k) X\;d\mu_K(k) , \;\;\; X\in E_\lambda.
\end{align*}
Recall that $\lbrace e_1^\lambda , \cdots , e^\lambda_{d_\lambda}\rbrace$   is  an orthonormal basis  of $ E_\lambda$, where  $ e_1^\lambda \in E_\lambda^l$ and  \\ $$d_\lambda=\dim E_\lambda .$$
\noindent {   It can be  shown that 
\[
\pi_\lambda(k)\, P_{l,\lambda}(X) = \chi_l(k)\, P_{l,\lambda}(X), 
  \]
for every \( k \in K \) and \( X \in E_\lambda \). 
Thus  $ P_{l,\lambda}(X) \in E_\lambda^l $.
}
\vspace{0.3em}
\noindent
Moreover, the operator \( P_{l,\lambda} \) is both self-adjoint and idempotent. Hence, it is the orthogonal projection of \( E_\lambda \) onto the subspace \( E_\lambda^l \). Consequently,
\[
\bigl(P_{l,\lambda}(e_j^\lambda),\, e_1^\lambda \bigr)_\lambda = 0,
\qquad \text{for } j = 2, 3, \dots, d_\lambda .
\]

{

\begin{lemma} \label{orthogonal projection 1}
Let $f \in L^1(U/K; \chi_l)$. Then:

\begin{itemize}
  \item[(i)] For each $\lambda \in \Lambda^+(U)$, we have
  \[
    \mathcal{F}(f)(\lambda) = \mathcal{F}(f)(\lambda) \, P_{l,\lambda}.
  \]

  \item[(ii)] For $\lambda \in \Lambda^+(U) \setminus \Lambda^+_l(U)$, we have
  \[
    \mathcal{F}(f)(\lambda) X = 0, \quad \text{for all } X \in E_\lambda.
  \]

  \item[(iii)] 
  \[
    \Tr\bigl(\mathcal{F}(f)(\lambda)\bigr) = \bigl( \mathcal{F}(f)(\lambda)e_1^\lambda,\, e_1^\lambda \bigr)_\lambda.
  \]
\end{itemize}
\end{lemma}

\begin{proof}
\begin{itemize}
  \item[(i)] 
  For $f \in L^1(U/K; \chi_l)$ and any $X \in E_\lambda$, we compute:
  \begin{align*}
    \mathcal{F}(f)(\lambda)\left(P_{l,\lambda}(X)\right)
    &= \int_U f(u) \, \pi_\lambda(u) \left( \int_K \chi_l(k^{-1}) \pi_\lambda(k) X \, d\mu_K(k) \right) d\mu_U(u) \\
    &= \int_K \int_U \chi_l(k^{-1}) f(u) \, \pi_\lambda(uk) X \, d\mu_U(u) \, d\mu_K(k) \\
    &= \int_K \left( \int_U f(uk) \pi_\lambda(uk) X \, d\mu_U(u) \right) d\mu_K(k) \\
    &= \int_K \left( \int_U f(u) \pi_\lambda(u) X \, d\mu_U(u) \right) d\mu_K(k) \\
    &= \int_U f(u) \pi_\lambda(u) X \, d\mu_U(u) = \mathcal{F}(f)(\lambda)(X).
  \end{align*}
  Hence,
  \[
    \mathcal{F}(f)(\lambda) = \mathcal{F}(f)(\lambda) \, P_{l,\lambda}.
  \]

  \item[(ii)] 
 Since $\lambda \not \in   \Lambda^+_l(U)$,  $E_\lambda^l =0 $.  Thus $$    \mathcal{F}(f)(\lambda) X = 0 \text{,  for all  } X \in E_\lambda.$$

  \item[(iii)] 

Since \( P_{l,\lambda} \) is the orthogonal projection onto the one-dimensional space spanned by \( e_1^\lambda \), for \( X \in E_\lambda \),
\[
P_{l,\lambda}(X) = (X, e_1^\lambda)_\lambda \, e_1^\lambda.
\]
Thus,
\begin{align*}
\mathcal{F}(f)(\lambda)(X)
&= \mathcal{F}(f)(\lambda)(P_{l,\lambda}(X)) \\
&= \mathcal{F}(f)(\lambda)\left( (X, e_1^\lambda)_\lambda \, e_1^\lambda \right) \\
&= (X, e_1^\lambda)_\lambda \, \mathcal{F}(f)(\lambda)(e_1^\lambda).
\end{align*}
It follows that
\[
\Tr\left( \mathcal{F}(f)(\lambda) \right) = \left( \mathcal{F}(f)(\lambda)(e_1^\lambda),\, e_1^\lambda \right)_\lambda.
\]

\end{itemize}
\end{proof}}

Denote by $(\cdot, \cdot)_{HS}$ the Hilbert–Schmidt inner product
 on $\End\left( E_{\lambda}\right) $,  i.e., 
 \begin{align}
\left( T, S\right) _{HS} & = \Tr\left( S^* \circ T \right) \nonumber\\
 & = \sum_{i=1}^{d_\lambda}\left( T e^\lambda_i , S e^\lambda_i \right) _{\lambda}, \label{Hilbert-Schmidt inner product} \nonumber 
 \end{align}
where $T, S \in  \End\left(E_{\lambda}\right)$, $d_\lambda=\dim E_\lambda$, $ \left\lbrace e^\lambda_i \right\rbrace_{i=1}^{d_\lambda} $ 
an orthonormal  basis of $E_{\lambda}$ and  $S^*$ the adjoint of $S$. It can be proved that this definition is independent of the choice of the orthonormal basis of $E_\lambda$. Let  $ \left\Vert . \right\Vert_{HS}$ the corresponding norm.
 
{   \begin{theorem}[Plancherel Theorem]  Let $f , f_1, f_2\in L^2 \left(U/K, \, \mathbb{L}_{\chi_l}\right)$. Then 
 \begin{align*}
\left( f_1,f_2 \right)_{L^2(U)}&= \sum _{\lambda\in \Lambda_l^+\left(U \right)} d_\lambda \left( \mathcal{F}(f_1 ) (\lambda), \mathcal{F} (f_2)(\lambda) \right)_{HS}.  &  
\end{align*}
In particular,
$$
\left\|f \right\|_{L^2\left( U\right) }^2 =\sum _{\lambda\in \Lambda_l^+\left(U \right)} d_\lambda \left\|  \mathcal{F}(f) (\lambda)\right\|_{HS}^2.
$$
\end{theorem}
}

Suppose that $\mu_{a_1,\dots ,a_r, \, \chi_l}$ is absolutely continuous with respect to the Haar measure $\mu_U$ on $U$, and denote by $f_{a_1,\dots ,a_r,\chi_l}$  the Radon-Nikodym derivative of  $\mu_{a_1,\dots ,a_r, \, \chi_l}$   with respect to $\mu_U$, i.e., $$ f_{a_1,\dots ,a_r,\chi_l} d\,\mu_U=d\,\mu_{ a_1,\dots ,a_r, \, \chi_l}.
$$
{   The following result will be used in the proof of Proposition \ref{Sobloev norm estimation}.}
\begin{lemma}\label{fouriour1 }
	$$
	\mathcal{F}\left(  f_{a_{1},...,a_{r},\chi_l}\right)(\lambda)e_1^\lambda = \left(  \mathcal{F}\left(  f_{a_{1},...,a_{r},\chi_l}\right)(\lambda)e_1^\lambda,e_1^\lambda \right)_\lambda e_1^\lambda.
	$$
\end{lemma}
\begin{proof}
	Note first that for each $k \in K$, the measures $f_{a_1,\dots ,a_r,\chi_l}\left(ku \right)d\mu_U\left( u\right)  $ and \\  $ \chi_l(k)^{-1} f_{a_1,\dots ,a_r,\chi_l}\left(u \right)d\mu_U\left( u\right)  $ are equal, i.e.,
	\begin{equation}\label{mesure-k-invariant}
	f_{a_1,\dots ,a_r,\chi_l}\left(ku \right)d\mu_U\left( u\right)  =  \chi_l(k)^{-1}  f_{a_1,\dots ,a_r,\chi_l}\left(u \right)d\mu_U\left( u\right) 
	\end{equation}
	for if $\phi \in \C\left(U \right) $, then 
 	\begin{scriptsize}
 	\begin{align*}
	\int_U\phi\left(u \right) f_{a_1,\dots ,a_r,\chi_l}\left(ku \right)\,d\mu_U\left( u\right)&= \int_U\phi\left(k^{-1}u \right) f_{a_1,\dots ,a_r,\chi_l}\left(u \right)\, d\mu_U\left( u\right)\\
	&=\int_K \cdots\int_K \phi\big(k^{-1}k_1a_1k_2\cdots k_ra_rk_{r+1}\big)\chi_l \left(k_1\cdots k_{r+1} \right)^{-1}d\mu_K(k_1) \cdots d\mu_K(k_{r+1})\\
	&=\int_K \cdots\int_K \phi\big(k_1a_1k_2\cdots k_ra_rk_{r+1}\big)\chi_l\left( k k_1\cdots k_{r+1} \right)^{-1}d\mu_K(k_1) \cdots d\mu_K(k_{r+1})\\
	&=  \int_U  \chi_l(k)^{-1} \phi\left(u \right) f_{a_1,\dots ,a_r,\chi_l}\left(u \right)d\mu_U\left( u\right).
	\end{align*}
 	\end{scriptsize}
 	
Therefore, from $\left(\ref{mesure-k-invariant} \right) $, we deduce that	
\begin{align*}
\pi_{\lambda}(k)\bigg[ 	\mathcal{F}\left(  f_{a_{1},...,a_{r},\chi_l}\right)(\lambda)e_1^\lambda\bigg] &=\int_U f_{a_{1},...,a_{r}}\left(u \right) \pi_{\lambda}\left(ku \right)e_1^\lambda \,d\mu_U \left(u \right)\\
&= \int_U f_{a_{1},...,a_{r}}\left(k^{-1}u \right) \pi_{\lambda}\left(u \right)e_1^\lambda \,d\mu_U \left(u \right) \\
&=\int_U  \chi_l(k) f_{a_{1},...,a_{r}}\left(u \right) \pi_{\lambda}\left(u \right)e_1^\lambda \,d\mu_U \left(u \right)\\
&=  	 \chi_l(k)\mathcal{F}\left(  f_{a_{1},...,a_{r},\chi_l}\right)(\lambda)e_1^\lambda .
	\end{align*} 
	{     Thus $ 	\mathcal{F}\left(  f_{a_{1},...,a_{r},\chi_l}\right)(\lambda)e_1^\lambda\in E_\lambda^l$.}   Since $\dim E_{\lambda}^l=1$, we deduce that
	$$
		\mathcal{F}\left(  f_{a_{1},...,a_{r},\chi_l}\right)(\lambda)e_1^\lambda=c\, \,e_1^\lambda,
	$$
	for some constant $c$, which is easily seen to be $\left(  \mathcal{F}\left(  f_{a_{1},...,a_{r},\chi_l}\right)(\lambda)e_1^\lambda,e_1^\lambda \right)_\lambda$. Hence the Lemma.  
\end{proof}

\begin{proposition} \label{trace-spherical-proposition}
\begin{align}\Tr(\F\left(f_{a_1,\dots ,a_r,\chi_l}\right)(\lambda)) = \prod\limits_{i=1}^{r}
	\psi_{\lambda,l}\left(a_i^{-1}\right).\label{trace-spherical}
\end{align}

\end{proposition} 
\begin{proof}
 Observe that 
\begin{small}
\begin{align}
\hspace*{-0.5 cm} \Tr(\F\left(f_{a_1,\dots ,a_r,\chi_l}\right)(\lambda))&=\left(\F\left(f_{a_1,\dots ,a_r,\chi_l}\right)(\lambda)e_1^\lambda, e_1^\lambda\right)_{\lambda}\;\;\;\;\;\;\;\;\;\; \;\;\;\;\;\;\;(\text{By Lemma }  \ref{orthogonal projection 1} )
              \nonumber  \\&=\int_U f_{a_1,\dots ,a_r,\chi_l}\left( u\right)  \left(\pi_\lambda \left(u \right)  e_1^\lambda,e_1^\lambda\right)_{\lambda} d\mu_U \left(u \right)\;\;\;\;\;\;\; 
                 \nonumber\\&=\int_U f_{a_1,\dots ,a_r,\chi_l}\left( u\right)   \psi_{\lambda,l}(u^{-1}) d\mu_U \left(u \right)\nonumber \\&=\int_U  \overline{\psi_{\lambda,l}(u)} \; d\mu_{ a_1,\dots ,a_r, \, \chi_l}(u)
              \label{inverse equation 1}   \\&=\int_K  \dots\int_K\overline{ \psi_{\lambda,l}\left(k_1a_1k_2\cdots k_ra_rk_{r+1} \right)}\chi_l(k_1 \dots k_{r+1})^{-1} d\mu_{K}\left(k_1 \right)  \cdots d\mu_{K}\left( k_{r+1}\right)
            \nonumber    \\&=\int_K  \dots\int_K \psi_{\lambda,l}\left((k_1a_1k_2\cdots k_ra_rk_{r+1})^{-1} \right)\chi_l(k_1 \dots k_{r+1})^{-1} d\mu_{K}\left(k_1 \right)  \cdots d\mu_{K}\left( k_{r+1}\right) 
             \label{inverse equation 2}  \\&=\int_K  \dots\int_K \psi_{\lambda,l}\left( k_{r+1}^{-1}a_r^{-1}k_r^{-1}\cdots k_2^{-1}a_1^{-1}k_{ 1}^{-1} \right)\chi_l(k_1 \dots k_{r+1})^{-1} d\mu_{K}\left(k_1 \right)  \cdots d\mu_{K}\left( k_{r+1}\right)
               \nonumber \\&=\int_K  \dots\int_K \psi_{\lambda,l}\left(  a_r^{-1}k_r^{-1}\cdots k_2^{-1}a_1^{-1}  \right)\chi_l(k_2 \dots k_{r})^{-1} d\mu_{K}\left(k_1 \right)  \cdots d\mu_{K}\left( k_{r+1}\right), \nonumber
\end{align} 
\end{small}
where \eqref{inverse equation 1} and \eqref{inverse equation 2} are obtained by using the fact that $ \overline{\psi_{\lambda,l}(u)}=\psi_{\lambda,l}\left(u^{-1}\right)$ for all  $u\in U$.\\

Since 
\begin{align*}
\int_K \psi(u_1ku_2)\chi_l(k)d\mu_K(k) = \psi(u_1)\psi(u_2), \text{ for all } u_1,u_2\in U,
\end{align*}
we infer that
 \begin{align} \nonumber 
 \hspace{-0.5cm} \Tr(\F\left(f_{a_1,\dots ,a_r,\chi_l}\right)(\lambda)) &=\int_K  \dots\int_K \psi_{\lambda,l}\left(  a_r^{-1}k_r^{-1}\cdots k_2^{-1}a_1^{-1}  \right)\chi_l(k_2 \dots k_{r})^{-1} d\mu_{K}\left(k_2 \right)  \cdots d\mu_{K}\left( k_{r }\right) 
              \\ \nonumber
               &= \psi_{\lambda,l}(a_1^{-1})\psi_{\lambda,l}(a_r^{-1}) \\ \nonumber
               &\;\;\; \times \int_K  \dots\int_K \psi_{\lambda,l}\left(  a_{r-1}^{-1}k_{r-1}^{-1}\cdots k_3^{-1}a_2^{-1}  \right)\chi_l(k_3 \dots k_{r-1})^{-1} d\mu_{K}\left(k_3 \right)  \cdots d\mu_{K}\left( k_{r-1 }\right) \\ \nonumber
                &\;\;\;\;\vdots \\ \nonumber
               & = \prod\limits_{i=1}^{r}
	\psi_{\lambda,l}\left(a_i^{-1}\right).
\end{align}
 Hence the Proposition.
  
\end{proof}

Let $\Delta$ denote the Casimir operator, $\langle \cdot, \cdot \rangle$ 
the inner product on $\mathfrak{a}^*$ induced by the Killing form of 
$\mathfrak{u}$, and let 
\[
  \kappa_\lambda = \langle \lambda, \lambda + 2\rho \rangle
\]
be the Casimir constant, where $2\rho $ is the sum of all positive roots and $\lambda \in \mathfrak{a}^*$.

Denote by $H^s(U)$ the Sobolev space of functions in 
$L^2(U)$ whose weak derivatives up to order $s$ are in $L^2(U)$, and let 
$\|\cdot\|_{H^s(U)}$ denote the corresponding Sobolev norm. For more details, 
see \cite{anchouche2}. 
\begin{proposition}\label{Sobloev norm estimation}
With the notation above, we have
	\begin{align*}
	\left\Vert  f_{a_{1},...,a_{r},\chi_l}\right\Vert _{H^{s}\left(U\right)}^{2}%
	=\sum\limits_{\lambda\in \Lambda_l^+\left(U \right)
		}d_\lambda\left(  1+\kappa_\lambda\right)  ^{s}\prod\limits_{i=1}^{r}\left\vert
	\psi_{\lambda,l}\left(a_i\right)  \right\vert ^{2}. 
\end{align*}		
\end{proposition} 
\begin{proof}
	Note first that 
	\begin{align}
	\left( \mathcal{F}\left(  f_{a_{1},...,a_{r},\chi_l}\right)(\lambda) , \mathcal{F} ( f_{a_{1},...,a_{r},\chi_l})(\lambda) \right)_{HS} &= \sum_{i=1}^{d_\lambda} \left( \mathcal{F}\left(  f_{a_{1},...,a_{r},\chi_l}\right)(\lambda)e_i^\lambda, \mathcal{F}\left(  f_{a_{1},...,a_{r},\chi_l}\right)(\lambda)e_i^\lambda\right)_\lambda
	\nonumber\\&= \left( \mathcal{F}\left(  f_{a_{1},...,a_{r},\chi_l}\right)(\lambda)e_1^\lambda, \mathcal{F}\left(  f_{a_{1},...,a_{r},\chi_l}\right)(\lambda)e_1^\lambda\right)_\lambda
	\nonumber\\&=\left\Vert \mathcal{F}\left(  f_{a_{1},...,a_{r},\chi_l}\right)(\lambda)e_1^\lambda \right\Vert_\lambda^2             
	 \nonumber\\&=\left\Vert \left(  \mathcal{F}\left(  f_{a_{1},...,a_{r},\chi_l}\right)(\lambda)e_1^\lambda,e_1^\lambda \right)_\lambda e_1^\lambda\right\Vert_\lambda^2 \nonumber 
	  \\
	&=\left\vert\Tr\left(\mathcal{F}\left(  f_{a_{1},...,a_{r},\chi_l}\right)(\lambda)\right) \right\vert^2. & \label{Fourier transform norm to trace}
	\end{align} 
{  	Equation (\ref{Fourier transform norm to trace}) is a consequence of Lemma  \ref{orthogonal projection 1} and Lemma \ref{fouriour1 }.} \\

Combining $\eqref{trace-spherical}$ and  $\eqref{Fourier transform norm to trace}$,   we obtain
\begin{align}
\left\Vert  f_{a_{1},...,a_{r},\chi_l}\right\Vert _{H^{s}\left( U\right) }^2 & = \left\Vert (I-\Delta) ^{\frac{s}{2}}  f_{a_{1},...,a_{r},\chi_l}   \right\Vert^2_{L^2(U)} \nonumber \\
&= \left( (I-\Delta) ^{s}  f_{a_{1},...,a_{r},\chi_l} ,  f_{a_{1},...,a_{r},\chi_l} \right)_{L^2(U)}\nonumber \\ 
&= \sum _{\lambda\in \Lambda_l^+\left(U \right)} d_\lambda \left( \mathcal{F}\left((I-\Delta) ^{s}  f_{a_{1},...,a_{r},\chi_l}\right) (\lambda), \mathcal{F} ( f_{a_{1},...,a_{r},\chi_l})(\lambda) \right)_{HS}\nonumber \\
&= \sum _{\lambda\in \Lambda_l^+\left(U \right)} d_\lambda \left((1+\kappa_{\lambda})^{s} \mathcal{F}\left(  f_{a_{1},...,a_{r},\chi_l}\right)(\lambda) , \mathcal{F} ( f_{a_{1},...,a_{r},\chi_l})(\lambda) \right)_{HS} \nonumber\\
&= \sum _{\lambda\in \Lambda_l^+\left(U \right)} d_\lambda (1+\kappa_{\lambda})^{s}  \left\vert\Tr\left(\mathcal{F}\left(  f_{a_{1},...,a_{r},\chi_l}\right)(\lambda)\right) \right\vert^2 \nonumber\\
&= \sum _{\lambda\in \Lambda_l^+\left(U \right)} d_\lambda (1+\kappa_{\lambda})^{s}\prod_{i=1}^r \left\vert\psi _{\lambda,l}(a_i)\right\vert^2. 
\label{trace 3} \nonumber 
\end{align}

\end{proof}

\section{Restricted Roots on Complex Grassmannians}\label{Complex Grassmannians}

Let $p$ and $q$ be two integers such that $p \geq q \geq 2,\; n = p + q,$ and let \begin{align*}
I_{p,q}=\begin{pmatrix}
I_p&0\\
0&-I_q
\end{pmatrix},
\end{align*}
where $I_n$ is the $n\times n$-identity matrix.    Consider the non-compact Lie group  \begin{align*}
G= SU(p,q) = \left\{   g\in SL(p+q,\mathbb{C} ) \mid g^*I_{p,q}g=I_{p,q} \right\}.
\end{align*}
Let $\mathfrak{g}=  \mathfrak{su}(p,q) $ denote  the  Lie algebra of  $SU(p,q)$, then {  \begin{align*}
\mathfrak{su}\left(  p,q\right) 
&  =\left\{  A\in M_{p+q}(\mathbb{C})\mid A^{\ast}I_{p,q}+I_{p,q}A=0, \Tr\left(  A\right)  =0\right\}.
\end{align*}}
Let $\mathfrak{su}(p,q) = \mathfrak{k}+\mathfrak{p}$ be a Cartan decomposition of $\mathfrak{su}(p,q)$, with 
\begin{align*}
\mathfrak{k}=\left\{ \left(
\begin{array}
[c]{cc}%
A & 0\\
0 & B
\end{array}
\right)  \mid A\in\mathfrak{u}\left(  p\right)  ,\text{ }B\in\mathfrak{u}%
\left(  q\right)  \text{ and }\Tr\left(  A\right) +\Tr\left( B\right)  =0\right\}
\end{align*}

and \begin{align*}
\mathfrak{p}=\left\{  \left(
\begin{array}
[c]{cc}%
0 & Z\\
\overline{Z}^{T} & 0
\end{array}
\right)  \mid Z\in M_{p,q}\left(\mathbb{C}\right)  \right\}.
\end{align*} Let $ \mathfrak{u}= \mathfrak{k}+\sqrt{-1}\mathfrak{p}$ be the compact real form of $\mathfrak{sl}(p+q, \mathbb{C})$,  the complexification of $\mathfrak{su}(p,q)$. Thus $ U=SU(p+q)$ is the lie group with lie algebra $\mathfrak{u}$. \\ Take $K = S\left(U(p)\times U(q) \right)$. Then the Lie algebra of $K$ is $\mathfrak{k}$.
  
Let $\mathfrak{a}$ be a maximal abelian subspace of $\mathfrak{p}$. For $ T= (t_1, t_2, \dots , t_q)\in \mathbb{R}^q$ we  choose $\mathfrak{a}$ to be all  $(p+q)\times (p+q)$-matrices of the form \[ H_T =\left(\begin{array}
[c]{ccc:c:ccc}
&  &  &  &  &  & t_{1}\\
& 0 &  & 0 &  & \reflectbox{$\ddots$} & \\
&  &  &  & t_{q} &  & \\
\hdashline & 0 &  & 0 &  & 0 & \\
\hdashline &  & t_{q} &  &  &  & \\
& \reflectbox{$\ddots$} &  & 0 &  & 0 & \\
t_{1} &  &  &  &  &  &
\end{array}
\right).
\]

Therefore we can identify $ \mathfrak{a}$ and $ \mathfrak{a}^*$, the dual of $\mathfrak{a}$,  with $ \mathbb{R}^q$.  Let $\alpha_i \in \mathfrak{a}^*$ 
be such that \begin{align*}
 \alpha_i\left(H_{(t_1, \dots , t_q)}\right)= t_i.
\end{align*}

Let $\Sigma=\Sigma(\mathfrak{g},\mathfrak{a}),$ be the set of restricted roots, so we have 
\begin{align*}
\Sigma= \left\{ \pm\alpha_i,   \pm 2\alpha_{i}, (1\leq i\leq q), \text{ and } \pm\left( \alpha_{i}\pm\alpha_{j}\right) , ( 1 \leq i< j \leq  q)\right\},
\end{align*}
  with multiplicities
\[
m_{\alpha_i}=2k,\;\; m_{2\alpha_i}= 1, \text{ and } m_{\alpha_i \pm \alpha_j}=2,
\]  where $k=p-q$. Let $\mathsf{C}^+$ be a  Weyl chamber in $\mathfrak{a}$ such that    
\begin{align*}
\mathsf{C}^+ = \left\{ H_{(t_1, \dots, t_q)} \in \mathfrak{a} \,\middle|\, t_1 > t_2 > \cdots > t_q > 0 \right\},
\end{align*} so   the corresponding system of
positive restricted roots $\Sigma^+$ consists of \begin{align*}
 \alpha_i,    2\alpha_{i}, (1\leq i\leq q), \text{ and } \left( \alpha_{i}\pm\alpha_{j}\right) , ( 1 \leq i< j \leq  q). 
\end{align*} 
Let \begin{align*}
\rho &= \frac{1}{2} \sum_{\alpha \in \Sigma^+} m_\alpha \alpha.
\end{align*} 
 Then \begin{align*}
\rho &=  \frac{1}{2}\left( \sum_{i=1}^q 2k\alpha_i + \sum_{i=1}^q 2\alpha_i + \sum_{i=1}^q 4 (q-i) \alpha_i    \right)\\&= \sum_{i=1}^q (k + 1 + 2 (q-i)) \alpha_i  .
\end{align*} 
By (\cite{Schlichtkrull},Proposition 7.1, Theorem 7.2; \cite{Ho-Olafsson}, Theorem 3.1), we have 

\begin{align*} 
\Lambda_l^+\left(U \right)=\Bigg\{\lambda= \sum_{i=1}^q \lambda_i\alpha_i \in \mathfrak{a}^*\Bigg|\;
\lambda_i - \lambda_j \in 2\mathbb{Z}^+(1\leq i< j \leq q),\; \lambda_1\in |l|+2\mathbb{Z}^+ \Bigg\}.
\end{align*}
Hence we have  
\begin{align}
\Lambda_l^+\left(U \right)=\left\lbrace\lambda= ( 2m_1+|l|, 2m_2+|l| , \dots , 2 m_q+|l|)\mid m_i\in \mathbb{Z} , m_1\geq m_2 \geq \dots \geq m_q\geq 0\right\}. \nonumber  \label{highest weight su(p,q)}
\end{align}
 An explicit formula for  the $\chi_l$-spherical functions on the complex Grassmannians $(SU(p,q)/S(U(p)\times U(q))$ was obtained in (\cite{Alhashami2}). More precisely, let $$P_n^{(\alpha, \beta)} (x) = \frac{(\alpha+1)_n}{n!} {}_2F_1\left( -n, n+\alpha+\beta+1, \alpha+1; \frac{1-x}{2}\right)$$ be  the  Jacobi polynomial of degree $n$,   where ${}_2F_1(\; \cdot\; ,\;\cdot\;,\; \cdot\; ;  \;\cdot\;  )$ is the Gauss hyper-geometric function, 
and let $$\tilde{P}_{n,l}\left(\cos(2x)\right)=\frac{P_n^{(k, |l|)} (\cos(2x))}{P_n^{(k,|l|)} (1)} = {}_2F_1\left(n+k+|l|+1,-n,k+1; \sin^2(x) \right).$$  Then we have    
\begin{theorem}[\cite{Alhashami2}]\label{Spherical function}
Let  $ \lambda\in \Lambda_l^+\left(U \right)$ such that  $$ \lambda= ( 2m_1+|l|, 2m_2+|l| , \dots , 2 m_q+|l|), m_i\in \mathbb{Z} , m_1\geq m_2 \geq \dots \geq m_q\geq 0.$$ Then the $\chi_l$-spherical function $\psi_\lambda$ on $U=SU(p+q)$
 is given by \begin{align*} 
\psi_{\lambda,l}\left(\exp \left(\sqrt{-1}H_{(t_1,\dots,t_q)}\right)\right)=  \frac{ \displaystyle C \det\bigg[\left( \tilde{P}_{n_i,l}(\cos(2t_j) \right) \bigg]_{i,\, j}\;\;\prod_{i=1}^q \cos^{|l|}(t_i) }{\displaystyle\prod_{1\leq i<j\leq q} (c(n_i)-c(n_j))\prod_{1\leq i<j\leq q}(\cos (2t_i)-\cos(2t_j))}  ,
\end{align*} where $H_{(t_1,\dots,t_q)}\in \mathfrak{a}$,  $k=p-q$, $c(n_i) = n_i(n_i+|l|+k+1)$  and 
\begin{align*}
 C=  2^{\frac{1 }{2} q(q-1)}\prod_{j=1}^{q-1}\left[(k+j)^{q-j}j!\right]   .\end{align*}
\end{theorem}

 \noindent
 \noindent
Note that the $\chi_l$-spherical function $\psi_{\lambda,l}$ in 
Theorem~\ref{Spherical function} can be defined for all elements of $U$ 
by means of the decomposition 
\[
U \;=\; K\,\exp\!\bigl(i\,\mathfrak{a}\bigr)\,K
\quad (\text{see Theorem 7.39 in \cite{Lie groups beyond}}).
\]
In other words, since $\psi_{\lambda,l}$  is $\chi_l$-bi-invariant, for any $u \in U$, 
\[
u \;=\; k_1\,\exp\!\bigl(i\,X\bigr)\,k_2,
\quad \text{ for some }
k_1, \; k_2 \,\in\, K, \text{ and } X\in \mathfrak{a},
\]
we obtain
\[
\psi_{\lambda,l}(u)
\;=\;
\chi_l(k_1 k_2)^{-1}\,\psi_{\lambda,l}\bigl(\exp(i\,X)\bigr).
\]
\begin{corollary}
 \label{estimation}

Let  $ \lambda\in \Lambda_l^+\left(U \right)$ such that  $$ \lambda= ( 2m_1+|l|, 2m_2+|l| , \dots , 2 m_q+|l|), m_i\in \mathbb{Z} , m_1\geq m_2 \geq \dots \geq m_q\geq 0$$  and let $n_j=m_j+q-j$ and let
\[
u \;=\; k_1\,\exp\!\bigl(i\,H_{\left(t_{1},...,t_{q}\right)}\bigr)\,k_2,
\]
 where 
\(k_1, \,  k_2 \,\in\, K\),  and  \(H_{\left(t_{1},...,t_{q}\right)}\in \mathfrak{a}\).

It follows  for sufficiently large $n_j$ we have 
\[	
\left\vert \psi_{{\lambda,l}}\left(u
  \right)\right\vert= \left\vert \psi_{{\lambda,l}}\left(\exp\left(\sqrt{-1}H_{\left(t_{1},...,t_{q}\right)}
\right)  \right)\right\vert \leq \begin{cases}
   \frac{C\left(l, t_1, \dots, t_q \right) }{ {\prod\limits_{j=1}^{q}} n_j^{\frac{1}{2}(2p-q)}  }&\mbox{ if } m_q>0,  \\\\  \frac{C\left(l, t_1, \dots, t_q \right)}{{\prod\limits_{j=1}^{q-1}}{n_{j}}^{\frac{1}{2} (2p-q+3)}} &\mbox{ if }  m_q=0,
 \end{cases}
 \]
where $H_{(t_1,\dots,t_q)}\in \mathfrak{a}$ and  $C\left(l, t_1, \dots, t_q \right) $ is a constant depending on $t_1, \dots, t_q$ and $l$.
\end{corollary}
\begin{proof}
By Theorem \ref{Spherical function} and a similar argument as in \cite{Alhashami}.
\end{proof}

\section{Proof of the Main Theorem} \label{Proof of the Main Theorem}

\begin{lemma}
Let $(\pi_\lambda , E_\lambda)$ be an irreducible $\chi_l$-spherical representation of \\ $U=SU(p+q)$ with  highest weight
$\lambda =(2m_1+|l|,\dots ,2m_q+|l|)$ and $d_\lambda=\dim(E_\lambda)$.
Then  
\begin{footnotesize}
\begin{align*}
d_\lambda &=\begin{cases} \displaystyle \prod_{i=1}^q  \frac{\varphi\Big(2m_i+|l|+k+1+2(q-i); \frac{1}{2}(k-1)\Big)}{\varphi\Big(k+1+2(q-i); \frac{1}{2}(k-1)\Big)}       \prod_{i=1}^q \frac{2m_i+|l|+k+1+2(q-i)}{ k+1+2(q-i)}  
\\\;\;\;\;\times \displaystyle \prod_{1\leq i<j\leq q } \left(\frac{ m_i+m_j+|l|+ k+2q-(i+j) }{k+2q-(i+j)}\right)^2
 \prod_{1\leq i<j\leq q } \left(\frac{ m_i-m_j+j-i}{j-i} \right)^2   \;\;  \text{ if } p\neq q,\\ \\\\\displaystyle \prod_{i=1}^q \frac{ 2m_i+|l|+ 1+2(q-i)}{ 1+2(q-i)}
\prod_{1\leq i<j\leq q } \left(\frac{m_i+m_j+|l|+ 1+2q-(i+j) }{1+2q-(i+j)}\right)^2
 \\  \;\;\;\;\times  \displaystyle   \prod_{1\leq i<j\leq q } \left(\frac{m_i-m_j+j-i}{j-i} \right)^2 \hspace{6.2 cm} \text{ if }p=q.
 \end{cases}
\end{align*}  
\end{footnotesize}

\end{lemma}
\begin{proof}
Let
\begin{align*}
\Phi(x,y;m)= \frac{\varphi(x+y;\, m)}{\varphi(y;\, m)},
\end{align*}
with 
\begin{align*}
\varphi(x;t)= (x - t)(x - t + 1)\cdots(x + t)
\end{align*} 
Let
\begin{align*}
W(x,y,m,1)= \Phi(x,y;0)\left[\Phi\left(x,y;\frac{1}{4}m-\frac{1}{2}  \right) \right]^2,
\end{align*}

\begin{align*}
W(x,y,m,0)= \begin{cases} 
\Phi(x,y;0), \;\;\;\;\;\;\; \text{ if }m=1,\\
\Phi(x,y;0)^2\;\;\;\;\;\;\; \text{ if } m=2.
\end{cases}
\end{align*}

Then by  (\cite{Goodman}, Theorem 2.1), we have
\begin{align*}
d_\lambda &= \prod_{\alpha\in \Sigma^+_0} W\Big( \langle \lambda, \alpha\rangle, \langle \rho ,\alpha\rangle ; m_{\alpha},m_{2\alpha}\Big),
\end{align*} 
where $\Sigma_0^+$ is the set of the indivisible positive roots.
For   $\lambda=( \lambda_1,\dots , \lambda_q)$,   \\$ \mu=( \mu_1,\dots , \mu_q) \in \mathfrak{a}^*$, and $n=p+q$, we have
\begin{align*}
\langle \lambda, \mu \rangle  = 4n \sum_{i=1}^q \lambda_i\mu_i. 
\end{align*} 

Therefore for $p\neq q$ we have 
\begin{footnotesize}
\begin{align*}
d_\lambda &= \prod_{\alpha\in \Sigma^+_0} W\Big( \langle \lambda, \alpha\rangle, \langle \rho ,\alpha\rangle ; m_{\alpha},m_{2\alpha}\Big)
\\&=  \prod_{i=1}^q W\Big({4n}(2m_i+|l|), {4n}[k+1+2(q-i) ];2k,1\Big) 
 \\ &\;\;\;\;\;\;\;\times   \prod_{1\leq i<j\leq q } W\Big( {8n}(m_i+m_j+|l|), {8n}(k+1+2q-(i+j)) ; 2,0\Big)  \prod_{1\leq i<j\leq q } W\Big( {8n}(m_i-m_j), {8n}(j-i) ; 2,0\Big) 
\\ &=  \prod_{i=1}^q \Phi\bigg( {4n}(2m_i+|l|), {4n}[k+1+2(q-i)];\frac{1}{2}\left( k-1\right)\bigg) ^2 \prod_{i=1}^q \Phi\Big( {4n}(2m_i+|l|), {4n}[k+1+2(q-i)]\Big)
\\&\;\;\;\;\times\prod_{1\leq i<j\leq q } \Phi\Big( {8n}(m_i+m_j+|l|), {8n}(k+2q-(i+j)) \Big)^2
 \prod_{1\leq i<j\leq q } \Phi\Big( {8n}(m_i-m_j),{8n}(j-i) \Big)^2
 \\&= \prod_{i=1}^q  \frac{\varphi\Big(2m_i+|l|+k+1+2(q-i); \frac{1}{2}(k-1)\Big)}{\varphi\Big(k+1+2(q-i); \frac{1}{2}(k-1)\Big)}       \prod_{i=1}^q \frac{2m_i+|l|+k+1+2(q-i)}{ k+1+2(q-i)}  
\\&\;\;\;\;\times\prod_{1\leq i<j\leq q } \left(\frac{ m_i+m_j+|l|+ k+2q-(i+j) }{k+2q-(i+j)}\right)^2
 \prod_{1\leq i<j\leq q } \left(\frac{ m_i-m_j+j-i}{j-i} \right)^2   .
\end{align*}
\end{footnotesize}
A similar argument can be applied for the case $p=q$.
\end{proof}
 
\begin{corollary} \label{dimention 1}Let 
\[
\lambda 
= \bigl(2m_1 + |l|,\; 2m_2 + |l|,\; \dots,\; 2m_q + |l|\bigr) 
\;\in\; \Lambda_l^+ (U).
\]
Then  
\[
d_\lambda \;\le\; \bigl(2m_1 + |l| + 1\bigr)^{q(2p-1)}\;\le\; C\left(l,p,q\right)\, n_1^{\,q(2p-1)},
\]
where \(n_1 = m_1 + q - 1\) and \(C\left(l,p,q\right)\) is a positive constant depending on \(l\), \(p\), and \(q\).

\end{corollary}
 
\begin{proof}
Suppose that \(p > q\). Then
 \begin{align*}
\frac{\varphi\Bigl(2m_i + |l| + k + 1 + 2(q-i);\tfrac{1}{2}(k-1)\Bigr)}
     {\varphi\Bigl(k + 1 + 2(q-i);\tfrac{1}{2}(k-1)\Bigr)}
&= 
\prod_{j=0}^{k-1}
\frac{2m_i+|l|+k+1+2(q-i)-\tfrac{1}{2}(k-1)+j}
     {k+1+2(q-i)-\tfrac{1}{2}(k-1)+j}\\[1ex]
&=
\prod_{j=0}^{k-1}
\left(\frac{2m_i+|l|}{\,k+1+2(q-i)-\tfrac{1}{2}(k-1)+j}+1\right)\\[1ex] 
&\le (2 m_1+|l|+1)^k,
\end{align*}
where $k+1+2(q-i)-\frac{1}{2}(k-1)+j \geq k+1+2(q-i)-\frac{1}{2}(k-1) > 1  $.

It follows that
\[
\begin{aligned}
&\prod_{i=1}^q \frac{\varphi\Bigl(2m_i + |l| + k + 1 + 2(q-i);\tfrac{1}{2}(k-1)\Bigr)}
                  {\varphi\Bigl(k + 1 + 2(q-i);\tfrac{1}{2}(k-1)\Bigr)}
\prod_{i=1}^q \frac{2m_i+|l|+k+1+2(q-i)}{k+1+2(q-i)}\\[1ex]
&\quad \times \prod_{1\leq i<j\leq q } \Biggl(\frac{m_i+m_j+|l|+k+2q-(i+j)}
      {k+2q-(i+j)}\Biggr)^2
\prod_{1\leq i<j\leq q } \Biggl(\frac{m_i-m_j+j-i}{j-i}\Biggr)^2\\[1ex]
&\le (2m_1+|l|+1)^{\,qk + q + q(q-1) + q(q-1)}\\[1ex]
&= (2m_1+|l|+1)^{\,q(p-q) + q + 2q(q-1)}\\[1ex]
&\le (2m_1+|l|+1)^{\,q(2p-1)}.
\end{aligned}
\]

A similar argument applies when \(p=q\).  
\end{proof}

 \begin{lemma} \label{Casimir operator estimate}
  Let \(\lambda = (2m_1 + |l|, \dots, 2m_q + |l|)\in \Lambda^+_l(U)\), with $\|\lambda\| \geq \|\rho\| $ then 
\begin{equation*}
\kappa_{\lambda}
    \le\; C\left(l\right) \,n_1^2.
\end{equation*} for some  positive constant \(C\left(l\right) \) depending on   $l$.
 \end{lemma}
 \begin{proof} Let \(\lambda = (2m_1 + |l|, \dots, 2m_q + |l|)\in \Lambda^+_l(U)\), by the Cauchy-Schwarz inequality, for \(\lambda = (2m_1 + |l|, \dots, 2m_q + |l|)\in \Lambda^+_l(U)\),
\begin{align*}
\kappa_{\lambda}
   \; &=\; \langle \lambda + 2\rho, \lambda\rangle 
   \\ &\le\; \|\lambda + 2\rho\|\;\|\lambda\|
    \\ &\le\; \|\lambda\|^{2} + 2\,\|\rho\|\;\|\lambda\| 
    \\ &\le\; 3 \|\lambda\|^{2}.
\end{align*}
Since all norms on a finite-dimensional vector space are equivalent, 
\begin{align*}
  \kappa_\lambda  \le\; 3 \|\lambda\|^{2}
   \;\le\; C_{1}\,\bigl(2m_1 + |l|\bigr)^2
   \;\le\; C_{1}\,\bigl(2n_1 + |l|\bigr)^2
   \;\le\; C\left(l\right)\,n_1^2,   
\end{align*}
for some positive constants  \(C_1 \) and  \(C\left(l\right)  \).
 \end{proof}
\begin{proposition}\label{prop2}
With the above notation, we have
\begin{align*}
\left\Vert f_{a_{1}, \dots, a_{r},\chi_l}\right\Vert _{H^{s}\left(U\right)}^{2}
&\leq    
              \sum_{\substack{\lambda\in \Lambda_l^+(U)\\\|\lambda\| < \|\rho\|}} 
   d_\lambda \,\bigl(1+\kappa_\lambda\bigr)^{s}\,\prod_{i=1}^{r}
   \bigl\vert \psi_{\lambda,l}(a_i)\bigr\vert^{2}  \\ &\hspace{0.5cm}+ C(l,a_1, \dots, a_r,p,q) \,\Biggl[\,
 \sum_{n =1}^{\infty}\,\frac{1}{\,n ^{\,r(2p-q) - 2s - q(2p-1)}} 
+ \sum_{n =1}^{\infty}\,\frac{1}{\,n ^{\,r(2p-q+3) - 2s - q(2p-1)}} 
\Biggr],
\end{align*}
where \(C(l,a_1, \dots, a_r,p,q)\) is a positive constant depending on $l$, \(a_1, \dots, a_r\), $p$ and  $q$. 
\end{proposition}

\begin{proof}

Using   Corollary~\ref{dimention 1},  Corollary~\ref{estimation}, Lemma \ref{Casimir operator estimate}   and Proposition \ref{Sobloev norm estimation}, we get
{  
 
\begin{footnotesize}
\begin{align*}
 \hspace{-5cm}\bigl\Vert & f_{a_{1}, \dots, a_{r},\chi_l}\bigr\Vert_{H^{s}(U)}^{2}\;\;\;\;\;\;  - \sum_{\substack{\lambda\in \Lambda_l^+(U)\\\|\lambda\| <\|\rho\|}} 
   d_\lambda \,\bigl(1+\kappa_\lambda\bigr)^{s}\,\prod_{i=1}^{r}
   \bigl\vert \psi_{\lambda,l}(a _i)\bigr\vert^{2}  
 \\ &=\; \sum_{\lambda\in \Lambda_l^+(U)}  
   d_\lambda \,\bigl(1+\kappa_\lambda\bigr)^{s}\,\prod_{i=1}^{r}
   \bigl\vert \psi_{\lambda,l}(a_i)\bigr\vert^{2}- \sum_{\substack{\lambda\in \Lambda_l^+(U)\\\|\lambda\| <\|\rho\|}} 
   d_\lambda \,\bigl(1+\kappa_\lambda\bigr)^{s}\,\prod_{i=1}^{r}
   \bigl\vert \psi_{\lambda,l}(a _i)\bigr\vert^{2}   
             \\& = \sum_{\substack{\lambda\in \Lambda_l^+(U)\\\|\lambda\| \geq \|\rho\|}}
   d_\lambda \,\bigl(1+\kappa_\lambda\bigr)^{s}\,\prod_{i=1}^{r}
   \bigl\vert \psi_{\lambda,l}(a _i)\bigr\vert^{2}
    \\&\le\;      C_1(l,a_1, \dots, a_r,p,q)\,\Biggl[\,\sum_{n_1>n_2>\dots>n_q\geq 1} \frac{n_1^{\,2s + q(2p-1)}}{\prod_{j=1}^q n_j^{\,r(2p-q)}} 
\;+\; \sum_{n_1>n_2>\dots>n_{q-1}\geq 1} \frac{n_1^{\,2s + q(2p-1)}}{\prod_{j=2}^{q-1} n_j^{\,r(2p-q+3)}} 
\Biggr]
\\ &\le\;  C_1(l,a_1, \dots, a_r,p,q)\,\Biggl[\,
\sum_{n_1>n_2>\dots>n_q\geq 1} \frac{1}{\,n_1^{\,r(2p-q)-2s - q(2p-1)}\,\prod_{j=2}^{q }n_j^{\,r(2p-q)}} 
\\
&\qquad\qquad\qquad\qquad
+\;\sum_{n_1>n_2>\dots>n_{q-1}\geq 1} \frac{1}{\,n_1^{\,r(2p-q+3)-2s - q(2p-1)} \,\prod_{j=2}^{q-1} n_j^{\,r(2p-q+3)}} 
\Biggr]
\\
&\le\;  C_1(l,a_1, \dots, a_r,p,q)\,\Biggl[\,
 \Biggl(\sum_{n_1=1}^{\infty}\,\frac{1}{\,n_1^{\,r(2p-q) - 2s - q(2p-1)}} 
 \Biggl) \Biggl(\sum_{n_2=1}^{\infty} \frac{1}{n_2^{\,r(2p-q)}}\Biggr)\cdots \Biggl(\sum_{n_q=1}^{\infty} \frac{1}{n_q^{\,r(2p-q)}}\Biggr)
\\
&\qquad\qquad\qquad\qquad\qquad\qquad 
+ \Biggl(\sum_{n_1=1}^{\infty}\,\frac{1}{\,n_1^{\,r(2p-q+3) - 2s - q(2p-1)}} \Biggl)
  \Biggl(\sum_{n_2=1}^{\infty} \frac{1}{n_2^{\,r(2p-q)}}\Biggr) \cdots \Biggl(\sum_{n_{q-1}=1}^{\infty} \frac{1}{n_{q-1}^{\,r(2p-q)}}\Biggr)
\Biggr],   
\\ &\leq C_1(l,a_1, \dots, a_r,p,q)\,  \Biggl[\sum_{k=1}^{\infty} \frac{1}{k^{\,r(2p-q)}}\Biggr]^{q-2}\Biggl[\,
 \sum_{n =1}^{\infty}\,\frac{1}{\,n ^{\,r(2p-q) - 2s - q(2p-1)}} 
\\
&\qquad\qquad\qquad\qquad\qquad\qquad\qquad\qquad\qquad\qquad\qquad\qquad\qquad
+ \sum_{n=1}^{\infty}\,\frac{1}{\,n ^{\,r(2p-q+3) - 2s - q(2p-1)}} \Biggr]
,               
\end{align*}
\end{footnotesize}}
   for some positive constant
\(C_1(l,a_1, \dots, a_r,p,q)\).
Since \(2p - q \ge p \ge 2\), we have \(r(2p-q) \ge 2\). Thus  \(\sum_{k=1}^{\infty} k^{-\,r(2p-q)}\)   converges. Hence the proposition.
\end{proof}

\begin{corollary}\label{cor-fund}
If  
\[
   r \;>\; \frac{1 + 2s + q\,(2p-1)}{\,2p - q\,},
\]
then the function \(f_{a_{1}, \dots, a_{r}, \chi_l}\) is in \(H^{s}(U)\).  In other words,
\[
   \bigl\Vert f_{a_{1}, \dots, a_{r}, \chi_l}\bigr\Vert_{H^{s}(U)} \;<\;\infty.
\]
\end{corollary}

\begin{proof}
Observe that
 \begin{align*}
 \sum_{\substack{\lambda\in \Lambda_l^+(U)\\\|\lambda\| < \|\rho\|}} 
   d_\lambda \,\bigl(1+\kappa_\lambda\bigr)^{s}\,\prod_{i=1}^{r}
   \bigl\vert \psi_{\lambda,l}(a _i)\bigr\vert^{2}
 \end{align*}  is a finite sum. Also 
\[
   \sum_{\,n =1 } 
   \frac{1}{n ^{r(2p-q) - 2s - q(2p-1)} }
\]
converges whenever
\[
   r \;>\; \frac{1 + 2s + q(2p-1)}{2p - q}.
\]
Moreover 
\begin{align*}
\sum\limits_{n=1} ^\infty \frac{1}{   n ^{r \left(2p-q+3\right)-2s -q(2p-1)}   }  
\end{align*}
  converges if
\[
   r \;>\; \frac{1 + 2s + q(2p-1)}{\,2p - q + 3\,}.
\]
Hence by Proposition \ref{prop2} the result follows.
\end{proof}

Let \(\nu\) be a positive integer and define
\begin{equation}\label{c( p,q,nu)} \nonumber
   C\!\bigl(p,q,\nu\bigr) 
   \;:=\; \max\!\Bigl\{
     \,\left\lfloor \tfrac{(p+q)^{2} + 2\nu + q(2p-1)}{\,2p - q\,}\right\rfloor + 1,\; 2pq 
   \Bigr\}.
\end{equation}

\begin{theorem}\label{main}
Let \(\nu\) be a positive integer. If 
\[
   r \;\ge\; C\!\bigl(p,q,\nu\bigr),
\]
then 
\[
   \frac{d\bigl(\mu_{a_1,\chi_l} \ast \cdots \ast \mu_{a_r,\chi_l}\bigr)}{\,d\mu_U}
   \;\;\in\; \C^{\,\nu}\bigl(SU(p+q)\bigr).
\]
\end{theorem}

\begin{proof}
Observe first that
\[
   \dim\Bigl(SU(n)/S\bigl(U(p)\times U(q)\bigr)\Bigr) \;=\; 2pq.
\] Thus by Theorem \ref{Rgozin Theorem}, $ \mu _{a_1, \dots , a_r, \, \chi_l}$ is absolutely continuous with respect to the Haar measure $\mu_U$ of  $U$,  i.e., $ f_{a_1,\dots ,a_r,\chi_l} d\,\mu_U=d\,\mu_{ a_1,\dots ,a_r, \, \chi_l}.$

Let \(r \ge C(p,q,\nu)\) be a positive integer. By definition of \(C(p,q,\nu)\), we can choose  \(\epsilon>0\) such that
\[
   r 
   \;>\; \frac{(p+q)^{2} + 2\nu + q(2p - 1) + 2\epsilon}{\,2p - q\,}.
\]
Then, by Corollary~\ref{cor-fund} (taking
\(\displaystyle s = \nu \;+\; \frac{(p+q)^2 - 1}{2} \;+\;\epsilon\)),
we have
\[
   f_{a_{1}, \dots, a_{r}, \chi_l} 
   \;\in\; H^s\bigl(SU(p+q)\bigr).
\]
 Therefore,  the Sobolev Embedding Theorem implies
\[
   H^s\bigl(SU(p+q)\bigr)
   \;\subseteq\; \C^{\nu}\bigl(SU(p+q) \bigr).
\]
Hence the result follows.
\end{proof}
 {

\end{document}
\begin{thebibliography}{99}

\bibitem{Alhashami}
M. Al-Hashami and B. Anchouche,
\emph{Convolution of orbital measures on complex Grassmannians},
J. Lie Theory \textbf{28} (2018), no.~3, 695--710.

\bibitem{Alhashami2}
M. Al-Hashami,
\emph{Berezin-Karpelevich formula for Chi-spherical functions on complex Grassmannians},
New Zealand J. Math. \textbf{50} (2020), 29--48.

\bibitem{anchouche-noncompact case}
B. Anchouche,
\emph{Regularity of the Radon-Nikodym derivative on non-compact symmetric spaces},
\url{https://doi.org/10.13140/RG.2.2.34657.97122} (August, 2018),
or \url{https://arxiv.org/abs/2107.12177}.



\bibitem{AG}
B. Anchouche and S. Gupta,
\emph{Convolution of Orbital Measures in Symmetric Spaces},
Bull. Aust. Math. Soc. \textbf{83} (2011), 470--485.

\bibitem{anchouche2}
B. Anchouche and S. Gupta,
\emph{Smoothness of the Radon-Nikodym derivative of a convolution of orbital measures
on compact symmetric spaces of rank one},
Asian J. Math. \textbf{22} (2018), no.~2, 211--222.

\bibitem{AGP}
B. Anchouche, S. K. Gupta, and A. Plagne,
\emph{Orbital measures on SU(2)/SO(2)},
Monatshefte f\"ur Mathematik \textbf{178} (2015), no.~4, 493--520.

\bibitem{Goodman}
S. Gindikin and R. Goodman,
\emph{Restricted roots and restricted form of Weyl dimension formula for spherical varieties},
\texttt{arXiv:1209.3002}, 2012.

\bibitem{Helgason3}
S. Helgason,
\emph{Differential geometry, Lie groups, and symmetric spaces},
Academic Press, New York, 1978.

\bibitem{Ho-Olafsson}
V. M. Ho and G. Olafsson,
\emph{Paley-Wiener theorem for line bundles over compact symmetric spaces and new estimates
for the Heckman-Opdam hypergeometric functions},
\texttt{arXiv:1407.1489}, 2014.

\bibitem{Lie groups beyond}
A. W. Knapp,
\emph{Lie groups beyond an introduction},
vol.~140, Birkh\"auser, Boston, 1996.

\bibitem{RaZonal}
D. L. Ragozin,
\emph{Zonal measure algebras on isotropy irreducible homogeneous spaces},
J. Funct. Anal. \textbf{17} (1974), no.~4, 355--376.

\bibitem{Schlichtkrull}
H. Schlichtkrull,
\emph{One-dimensional K-types in finite dimensional representations of semisimple Lie groups:
A generalization of Helgason's theorem},
Math. Scand. \textbf{54} (1984), no.~2, 279--294.

\bibitem{Shimeno}
N. Shimeno,
\emph{The Plancherel formula for spherical functions with a one-dimensional K-type on a simply connected
simple Lie group of Hermitian type},
J. Funct. Anal. \textbf{121} (1994), no.~2, 330--388.

\bibitem{Stromberg}
K. Stromberg,
\emph{A note on the convolution of regular measures},
Math. Scand. \textbf{8} (1960), 347--352.

\bibitem{Wolf3}
J. A. Wolf,
\emph{The geometry and structure of isotropy irreducible homogeneous spaces},
Acta Math. \textbf{120} (1968), no.~1, 59--148.

\end{thebibliography}
